\documentclass{article}

\usepackage[utf8]{inputenc}

\input{def.tex}

\title{The spectrum of a random operator is a random set}
\author{Gunnar Taraldsen\\
Norwegian University of Science and Technology}
\date{\today}

\begin{document}

\maketitle

\begin{abstract}
The theory of random sets is demonstrated to
prove useful for the theory of random operators.
A random operator is here defined by requiring the graph to be a random set.
It is proved that the spectrum and the set of eigenvalues of random operators are random sets.
These results seem to be a novelty even in the case of random bounded operators.
The main technical tools are given by the measurable selection theorem, 
the measurable projection theorem, 
and a characterisation of the spectrum by approximate
eigenvalues of the operator and the adjoint operator.
A discussion of some of the existing definitions of the concept of a
random operator is included at the end of the paper.

\vspace{1ex}

{\small


 \vspace{1ex}
 
\noindent 
{\it Keywords:} {\bf
 Random operators;
 Set-valued functions;
 General topics in linear spectral theory;  
 Random operators and equations;
 Stochastic integrals;
 Disordered systems
 
 }
}

\end{abstract}

\section{Introduction}

It will be assumed throughout that $(\Omega, \ceps, \pr)$ is a fixed underlying \Renyi space
which by definition is equivalent to a full conditional probability space as defined by
\citet[p.43]{RENYI}.
It will be convenient in the following to use the same symbol $\pr$ for both a \Renyi state, 
and also for a $\sigma$-finite measure $\pr$ that represents the \Renyi state.
A random quantity $X$ in a measurable space $(\Omega_X, \ceps_X)$
is a measurable mapping $X: \Omega \into \Omega_X$.
The law of $X$ is then defined by
$\pr_X (A) = \pr (X \in A) = \pr \{\omega \st X(\omega) \in A \}$
for all $A \in \ceps_X$.
The reader is hereby warned that the notation $(X \in A)$ is ambiguous.
It does not mean that $X$ is an element in $A$, but it denotes the event
$\{\omega \st X(\omega) \in A \}$.
This convention, and similar conventions for other
events determined by conditions on random quantities,
is used by \citet[p.22]{KOLMOGOROV} and other workers in probability \citep[p.11]{DOOB}.

If $\Theta$ is another random quantity, 
then a conditional \Renyi state $\pr^\theta_X$
is defined such that $\pr^\theta_X (A) = \pr (X \in A \st \Theta = \theta)$
as discussed by \citet{Taraldsen19conditional}.
This gives a basis for statistical inference for an unknown model parameter
$\theta$ given observed data $X$ including cases with 
random quantities such as
random vectors, random sets, random vector spaces, and random operators. 
The focus in the following will only be on the necessary prerequisite probability theory.  

Let $\Omega_X$ be the set of subsets of a topological space.
The space $\Omega_X$ is equipped with
the smallest $\sigma$-field $\ceps_X$
such that $\{F \in \Omega_X \st F \cap U \neq \emptyset\}$ is measurable
for all open sets $U$.
A random set $X$ is a measurable mapping $X: \Omega \into \Omega_X$.
The set $X_U = (X \cap U \neq \emptyset) = \{\omega \st X(\omega) \cap U \neq \emptyset \}$ is then
an event when $U$ is open.
A mapping $X: \Omega \into \Omega_X$ is a random set if and only if the closure $\overline{X}$ is a random set.
If $X$ is a singleton, then $X$ corresponds to a random element in $\Omega_X$.
Assume next that the topological space is a separable metric space,
and that $X$ is almost certainly non-empty and complete.
It follows then that
$X = \overline{\{Y_1, Y_2, \ldots \}}$ where the $Y_1, Y_2, \ldots$ are random elements in the metric space.
This measurable selection theorem is proved by \citet[p.67, Thm III.9]{CastaingValadier77}. 

Let $\Omega_X$ be the set of subspaces of a separable Hilbert space equipped with
the smallest $\sigma$-field $\ceps_X$ such that $\{F \in \Omega_X \st F \cap U \neq \emptyset\}$ is measurable
for all open sets $U$.
A random subspace $X$ in the Hilbert space is a measurable mapping $X: \Omega \into \Omega_X$.
A random subspace is hence a random set which is also a vector space.
Let $P (\omega)$ be a projection on a closed subspace $X(\omega)$ of a
separable Hilbert space for each $\omega \in \Omega$.
It follows from the measurable selection theorem that 
$P f$ is a random vector for all vectors $f$ if and only if $X$ is a random closed subspace.
The only if part follows since $P$ applied to a dense countable set gives a measurable selection.
The if part follows by constructing a random orthonormal basis from a dense measurable selection.

Let $\Omega_X$ be the set of linear operators from a separable Hilbert space into another separable Hilbert space.
A random operator $X$ is a mapping
$X: \Omega \into \Omega_X$ such that the graph $\Graph (X)$ is a random subspace.
The adjoint operator $X^*$ is a random closed operator since 
$\Graph (X^*) = (\Graph(-X)_{inv})^\perp$.
The claimed measurability follows from the previous characterisation of random closed subspaces
in terms of projections. 
The orthogonal complement operation $\perp$ above is defined relatively
to the random closed subspace
given by the product of the range Hilbert space
and the closure of the domain $\Dom (X)$.
This definition includes unbounded operators, and they need not be densely defined.

The measurable selection theorem applied to the graph $\Graph (X)$ of a random closed operator 
gives dense measurable selections for both the domain $\Dom (X)$
and range $\Ran (X)$ so they are random sets,
but they need not be random closed sets.
This gives a characterisation of random closed operators as infinite random matrix operators
given by a matrix of random complex numbers and two random orthonormal systems.
It follows from the closed graph theorem that a random closed operator
$X$ is a random bounded operator if and only if
the domain $\Dom (X)$ is a random closed subspace.
Random selfadjoint operators are characterised by randomness of their resolvent, 
and then equivalently in terms of their spectral families.
The latter generalises to a characterisation of random normal operators.

\section{The spectrum is a random set}

The results in the previous section, and in particular the proof of the randomness of the
adjoint, relied on the measurable selection theorem.
The results to follow depend on the measurable projection theorem:
Let $E$ be measurable in the product space of $\Omega$ with a complete separable metric space.
The projection of $E$ on $\Omega$ is an event if $\Omega$ is complete.
The latter means that all subsets of zero sets are also events.
A zero set $A$ is an event such that $\pr (A)$ = 0.

\citet[p.75, Thm III.23]{CastaingValadier77}  prove the measurable projection theorem
for the more general case of the
product of a complete probability space with a Suslin space: 
A Hausdorff space that is also the continuous image of a Polish space.
A consequence is that $F_E = (F \cap E \neq \emptyset)$ is an event for
any measurable set $E$ when $F$ is
a random closed set \citep[p.80, Thm III.30]{CastaingValadier77}.
Our stated generalisation to a
complete \Renyi space $\Omega$ follows since $\pr$ is equivalent to a probability measure in the
sense of having the same family of zero sets.

\begin{theo}
  \label{theo1}
The eigenvalue spectrum and the spectrum of a random closed operator in
a separable Hilbert space are random sets if the underlying \Renyi space is complete.
If, furthermore, the random closed operator is unitarily invariant with respect 
to a metrically transitive system of measurable transformations acting on the \Renyi space,
then both the closure of the eigenvalue spectrum and the spectrum are represented by fixed sets.
\end{theo}
\begin{proof}
  Let $\Sigma$ be the eigenvalue spectrum, let $G$ be the operator graph,
  and let $U$ be an open set in the complex plane.
It must be proven that $\Sigma_U = (\Sigma \cap U \neq \emptyset)$ is an event.
Observe first that $\Sigma_U = G_E$ with $E = \{(f, z f) \st \norm{f} = 1, z \in U \}$,
so it is sufficient to prove that $E$ is measurable.
The latter follows since 
$E$ is the set of pairs $(f,g)$ such that 
$\abs{\langle f,g \rangle} = \norm{f}\norm{g}$, $\norm{f} = 1$, and $\langle f,g \rangle \in U$.

The spectrum is the union of the approximate
spectrum and the complex conjugate of the approximate
spectrum of the adjoint operator.
The adjoint is a random operator and the union of two random sets is a random set, 
so it is sufficient to prove that the approximate spectrum is a random set.
The approximate spectrum is
$\Sigma = \{z \st \forall \epsilon > 0, \exists (f,g) \in G, \norm{f} = 1, \norm{z f - g} < \epsilon\}$.
The projection theorem gives that it is sufficient to prove that
$\{(\omega, z) \st z \in \Sigma (\omega)\}$ is measurable.
It is since it equals
$\cap_{N \ge 1} \cup_n \{(\omega, z) \st  \norm{z f_n (\omega) - g_n (\omega)} < \norm{f_n (\omega)}/N \}$
where $\{(f_n,g_n)\}$ is a dense measurable selection for $G$.

\vekk{
Now $\Sigma_U = \cup_{z \in U} G_{E(z)}$ with 
$E (z) = \cap_{n \ge 1} \{(f, g) \st \norm{f} = 1, \norm{z f - g} \le 1/n \}$.
The set $\Sigma_U$ is then an event since it is the projection of the jointly measurable set
$\{(\omega, z) \st G(\omega) \cap E(z) \neq \emptyset, z \in U\}$.
}

Both the spectrum and the eigenvalue spectrum are unitarily invariant.
The set of closed sets are separated by a countable family of measurable sets and the claimed existence of
fixed sets representing the spectra follows \citep[p.77]{Taraldsen92}.

\end{proof}

The above self-averaging of spectra for metrically transitive operators
is a generalization of similar results for random self-adjoint operators
obtained previously by \citet{Pastur73spectra,Pastur89spectral} and
\citet{KirschMartinelli82ergodic}.
The randomness of the spectra was proved by \citet{Taraldsen93ml},
but with a slightly more complicated argument.
The result for the spectrum holds in any Hilbert space if it is
assumed that there is a countable dense selection for the graph $G$.

The reader may wonder why the results are stated for an underlying \Renyi space
since the proofs are essentially unchanged from the more common assumption
of an underlying probability space.
The reason is that unbounded measures appear naturally in statistics
\citep{TaraldsenLindqvist13fidopt},
and they also appear naturally in the case of dynamic systems.
The latter is documented by for instance
\citet{BoshernitzanDamanikFillmanLukic19ergodic} who have results
on ergodic random self-adjoint operators for the case where the underlying $\pr$
is $\sigma$-finite.
\citet{OrlovSakbaevSmolyanov2016UnboundedRandomOperators} exemplify
that another possible generalization is given by replacing the
probability measure by a finitely additive normalized measure.
\citet{BieDruilhetSaintBie2019ImproperVsFinitely} show certain defects
related to use of finitely additive measures in statistics,
but the work of \citet{DEFINETTI} demonstrates for ever
the importance of considering also finitely additive measures.

\section{What is a random operator?}

\citet[p.1]{SKOROHOD} defines
\be{S1}
(A f)(t) = \int_0^t f(s) \, B(ds)
\ee
for $f \in \Hi = L^2 [0,1]$.
The integral is the stochastic integral
with respect to Brownian motion $B$.
This defines a stochastic process $A$ indexed by
$\Hi$ with values in $\Hi$ so that
$A (\alpha f + \beta g) = $
$\alpha A f + \beta A g$ holds almost surely for all
$\alpha, \beta \in \CompN$ and $f, g \in \Hi$.
Furthermore, $f_n \rightarrow f$ implies
$A f_n \rightarrow A f$ in probability.
Altogether, this shows that $A$ is a
strong random operator in the sense of
\citet[p.3]{SKOROHOD}.
\footnote{He uses a real Hilbert space, but it is convenient here to use a complex Hilbert space.}
It is important to note that the
null event in the above linearity demand depends on
the involved vectors and scalars.
It is hence not generally so that there exists an operator $A (\omega)$
for almost all $\omega$.

If, additionally,
there exists a random variable $K$ such that
$\norm{A f} \le K \norm{f}$ almost surely,
then $A$ is said to be a bounded
strong random operator.
\citet[p.7-8]{SKOROHOD} proves that the set of
bounded strong random operators can be identified with
the set of random bounded operators on $\Hi$.
The operator $A$ defined in equation~(\ref{eqS1})
is, however, not bounded, and demonstrates
that the set of random bounded operators is a strict subset
of the set of strong random operators.
This is related to the fact that the stochastic integral is
not given by a random integral with respect to a random
measure \citep[p.1]{Kallenberg2017RandomMeasuresTheory}.
\vekk{
\citet[p.19]{Kallenberg2005ProbabilisticSymmetriesInvariance} defines a continuous linear random
functional to be an abstract version of a stochastic integral.
The definition coincides\footnote{Kallenberg uses a real Hilbert space.}
with the concept of a strong random operator from $\Hi_1$ to $\Hi_2$
for the case where $\Hi_2 = \CompN$.
}

\vekk{In the following it will hence be natural to
introduce the term ``stochastic operator'' as a short form
for the term ``strong random operator in the sense of Skorohod''.}

\citet[p.3]{SKOROHOD} defines
also the more general concept of a
weak random operator $A$ as a continuous (in probability) stochastic process indexed by
$\Hi \times \Hi$ with values in $\CompN$ so that
$(f,g) \mapsto (f, A g)$ is almost surely a sesquilinear form \citep[p.1]{WEIDMANN}.
A strong random operator is a weak random operator.
The strong random operators are characterized as the weak random operators
with Fourier coefficients $(e_i,A g)$ in $l^2$ \citep[p.4]{SKOROHOD}. 
In this case a strong random operator is determined from a weak random operator
by defining $A g = \sum_i e_i (e_i, A g)$ given an orthonormal basis $\{e_i\}$.

Let $\Gamma$ be the Poisson point process on $\RealN^3$,
and let $\alpha$ be i.i.d. (non-zero) random variables independent of $\Gamma$.
\citet[Thm 7, p.6]{KaminagaMineNakano2019SelfadjointnessCriterionSchr} prove that
\be{RanPoisson}
H = -\Delta + \sum_{x \in \Gamma} \alpha(x) \delta_x
\ee
is a random selfadjoint operator on $\Hi = L^2 (\RealN^3)$,
and that the spectrum is almost surely equal to the real line. 
This beautiful result exemplify that the theory of
random Schr{\"o}dinger operators is still a most active field of research
\citep{DolaiKrishnaMallick2019RegularityDensityStates,STOLLMANN,Kirsch07invitation}. 
Many more examples of random Schr{\"o}dinger operators
are presented by \citet{CARMONA_LACROIX}
and \citet{PASTUR_FIGOTIN}.
These examples are all special cases of the definition used in Theorem~\ref{theo1},
but it is unclear in each specific case if they also can be seen
as weak random operators in the sense of \citet{SKOROHOD}. 
The particular case $H = -\Delta$ obtained by $\alpha = 0$
gives trivially a random selfadjoint operator,
but it is not a weak random operator since
the domain $D(H) \neq \Hi$.

\citet[p.30-32]{SKOROHOD} constructs a random selfadjoint operator
$\overline{A}$ from a
symmetric weakly random operator $A$ fulfilling
$\sum_j \abs{(e_j, A e_i)}^2 < \infty$.
In this case
$A e_j$ is defined for the particular orthonormal basis $\{e_j\}$.
\citet[p.36]{SKOROHOD} notes that this condition
holds in particular for any strong symmetric operator.
The construction gives a dense fixed subspace $D$ with
$<f, \overline{A} g> = (f,A g)$ for all $f,g \in D \subset D(\overline{A})$.
The result is a random selfadjoint operator with a fixed core.
This corresponds to the special
case of a random matrix operator defined from a fixed
orthonormal basis.
\citet[p.40-48]{SKOROHOD} uses
the corresponding spectral families to obtain
solutions to certain evolution equations, Schr{\"o}dinger equations,
and Fredholm type equations.
Furthermore, 
\citet[p.48-60]{SKOROHOD} generalizes the construction to include equations
with semi-bounded operators that are not necessarily symmetric.

More generally, it seems that the constructions can be used with a random orthonormal basis
to give a characterization of the weakly random symmetric operators
that have a corresponding selfadjoint operator.
This claim, and corresponding claims for semi-bounded operators, 
is left for future investigations.
An alternative strategy is to develop spectral theory from scratch
for weakly random operators.
The initial ingredient is given by the definition of products
of weakly random operators as formulated by
\citet[p.10]{SKOROHOD}, but possibly generalized by allowing
use of a random orthonormal basis.

\citet{ThangQuy17spectral}
presents results on strongly random operators and bounded strongly random operators
between and on separable
Banach spaces.\footnote{\citet[Def. 2.1]{ThangQuy17spectral}
  use the term random operator in stead of the term strongly random operator.}
It is in particular shown that a bounded strongly random operator
can be extended in the natural way to a continuous
linear operator from the set of Banach valued random variables to the
set of Banach valued random operators equipped with the topology from
convergence in probability.
\citet{ThangQuy17spectral} 
prove versions of the spectral theorem for bounded selfadjoint, and more generally
bounded normal  strongly random operators.
This supplement the results obtained by
\citet[p.40-48]{SKOROHOD}.

\citet{Hackenbroch09spectralTheory} generalizes the concept of
a strong random operator by replacing the index Hilbert space
by a fixed dense subspace $D$.
This includes then random operators with a common core,
and the Laplace operator $-\Delta$ is then trivially included in
the class of Hackenbrock random operators. 
Furthermore, \citet[Thm 1]{Hackenbroch09spectralTheory}
proves that the operators with a densely defined adjoint
have a unique closed extension with a a measurable selection.
This includes in particular symmetric operators.
The resulting random closed operator is a special case
of the closed operators covered by Theorem~\ref{theo1} since
the random operator has a fixed non-random core.
\citet{Gaspar18ranNormal} provides further links between different
concepts of random operators including normal Hackenbrock random operators and associated
random spectral families.
The classes of random operators considered are all restricted
by assuming that there is a dense nonrandom subspace $D$,
but the results are important generalizations of 
the results on selfadjoint extensions obtained by \citet[p.40-48]{SKOROHOD}.


Measurable fields of operators, and corresponding direct integrals, 
are fundamental in the theory of von Neumann algebras \citep{DIXMIER}.
In this context, the Hilbert space, including the inner product, may depend on $\omega$,
in addition to the operators dependence on $\omega$.
\citet[Lemma 12.1.3, p.332]{Schmudgen1990UnboundedOperatorAlgebras} provides
a measurable selection theorem for the graph of a
measurable field of closed operators.
This provides hence a generalization of the definition of a
random closed operator as used in Theorem~\ref{theo1}.
Inspection of the proof shows that its conclusions hold also in this more general case.
Another result is a direct integral of a random closed operator.
In the case of metrically transitive operators this gives a tool
for the determination of the spectrum, but this will not be investigated further here.
This, and more, give generalizations of results for
random Schr{\"o}dinger operators \citep{Kotani85support}, and random selfadjoint operators.
An explicit formula for the nonrandom spectrum of a metrically
transitive random normal operator from the topological
support of the operator is proved by \citet[p.84]{Taraldsen92}.

Let $\ran{D}$ be a mapping of the probability space into
the set of subsets of $\Hi_1$ such that 
$\ran{D}_{\{f\}} \dlik (f \in \ran{D})$ is
an event for all $f \in \Hi_1$.
A mapping 
$\ran{T} : \{(\omega,f) \st f \in \ran{D} (\omega)\} \into \Hi_2$ 
is a weakly $G$-random operator in the sense of
\citet{Zhdanok90invRandom}
if $G (\ran{T})$ is a random set,
and linear if $\ran{T} (\omega): \ran{D} (\omega) \into \Hi_2$
is linear for all $\omega$.
If the measurability demand on  $\ran{D}$ is removed, 
then $\ran{T}$ is a random operator in the sense of Theorem~\ref{theo1}.
We arrived at our definition independently of
the work of Zhdanok, 
and motivated by different applications.
The definitions are nonetheless similar, 
and this seems to indicate the naturalness of the definitions.
Use of the theory of random sets is a key ingredient
also in the arguments of \citet{Zhdanok90invRandom}.

Zhdanok obtains a Hille-Yosida theorem for the random generator
of a random strongly continuous semigroup.
This is related to the results of \citet[p.48-60]{SKOROHOD} on
semi-bounded operators that are not necessarily symmetric.
\citet{OrlovSakbaevSmolyanov2016UnboundedRandomOperators}
study random one-parameter semigroups using the Trotter-Lie product formula,
and introduce in particular a definition for the expectation of
the corresponding random unbounded generator.
This, and other examples indicated above,
demonstrate then the need for a theory
for random closed operators which are not necessarily selfadjoint nor normal.

\section{Conclusion}

It follows from the previous brief discussion that
there are many different definitions of the concept of a random operator,
and there are recent progress on relations between them.
Different applications require different concepts.
Applications of weak and strong random operators are well documented
by the books by \citet{SKOROHOD} and \citet{BHARUCHA_REID},
and later works referring to these two fundamental contributions. 
Applications involving the theory of random closed operators,
and in particular the theory of random Schr{\"o}dinger operators,
are similarly documented by the seminal works of
\citet{PASTUR_FIGOTIN} and
\citet{CARMONA_LACROIX}.

A theory based on the theory of random sets has been used here
based on initial work by \citet{Taraldsen92}.
The definition of a random operator used here was initially
inspired by the gap topology of \citet{KATO} based on the graph.
Links to other approaches have been indicated.
Further progress is likely to follow using recent results from
the theory of random sets \citep{Molchanov2017TheoryRandomSets},
and random measures \citep{Kallenberg2017RandomMeasuresTheory}.
The key result presented here is that the spectrum of a random closed operator
is in fact a random set in the sense
defined and studied originally by \citet{MATHERON}.
The results presented are, hopefully, convincing demonstrations of the
applicability of the theory of random sets to the theory of random operators.
\vekk{
It is also natural to consider topologies for closed sets, then also for closed operators,
and for random operators.
\citet{Taraldsen92} provides some initial results in this direction.
It is in particular demonstrated that strong resolvent convergence of selfadjoint operators is
equivalent with a corresponding natural topology for closed sets.
}

\bibliographystyle{chicago}
\bibliography{tex}

\end{document}